\documentclass{amsart}

\usepackage[T1]{fontenc}
\usepackage[utf8]{inputenc}
\usepackage{amssymb,mathtools}
\usepackage{microtype}
\usepackage[hidelinks]{hyperref}

\theoremstyle{plain}
\newtheorem{theorem}{Theorem}[section]
\newtheorem{lemma}[theorem]{Lemma}
\newtheorem{proposition}[theorem]{Proposition}
\newtheorem{corollary}[theorem]{Corollary}
\theoremstyle{remark}
\newtheorem{remark}[theorem]{Remark}
\numberwithin{equation}{section}

\newcommand{\numberthis}[1]{%
\refstepcounter{equation}%
\label{#1}%
\eqno(\theequation)}

\title[Equal-Phase Monotonicity for Jacobi--Radau Functions]
{Equal-Phase Monotonicity for Weighted Jacobi--Radau Functions and
Jacobi Lebesgue Constants}

\author{K. Castillo}

\address{CMUC, Department of Mathematics, University of Coimbra,
3000-143 Coimbra, Portugal}

\email{math@keniercastillo.com}

\author{P.-C. Hang}

\address{School of Mathematics and Statistics, Donghua University,
Shanghai 201620, People's Republic of China}

\email{mathroc618@outlook.com}

\date{\today}

\dedicatory{Dedicated to Professor Roderick S. C. Wong, whose contributions
have inspired us all.}



\begin{document}

\begin{abstract}
For the Jacobi family with parameters $(0,\beta)$, $\beta\geq-1/3$,
we prove a continuous comparison theorem for the associated weighted
Jacobi--Radau functions.  When two consecutive functions are
parametrised by the same Pr\"ufer phase, the function of higher degree
attains that phase closer to $\theta=0$ and has strictly
larger amplitude.  In particular, the moduli of all corresponding
relative extrema increase strictly with the degree.  The lower bound
$-1/3$ is sharp for this continuous statement: when
$-1<\beta<-1/3$, the amplitude inequality is reversed at sufficiently
small positive phases.  This local reversal does not determine the
optimal range for the discrete extremal inequalities.  For
$-1/3\leq\beta\leq0$, the positive Jacobi product formula identifies
the endpoint values of the Lebesgue functions with the global
Lebesgue constants, so $(\Lambda_n^{(0,\beta)})_{n\geq0}$ is
strictly increasing.
At $\beta=0$ the weighted
Radau functions reduce to $P_m^{(0,-1)}$.  We thereby recover the
theorem of Wong and Zhang and obtain, through an exact total-variation
formula, a short proof of the Qu--Wong theorem on Legendre Lebesgue
constants.  The representation and the termwise comparison together
realise the alternative-expression approach proposed by Qu and Wong,
without asymptotic expansions, error bounds, or finite
numerical verification.  The proof is based on the equal-phase
Pr\"ufer architecture developed in the first author's earlier
preprint \cite{Castillo2026}; it applies that architecture to the
problem posed by Qu and Wong.
\end{abstract}

\maketitle

\section{Introduction}

Let $P_n^{(\alpha,\beta)}$ denote the Jacobi polynomial of degree $n$
in the standard normalisation
$$
P_n^{(\alpha,\beta)}(1)
=
\frac{(\alpha+1)_n}{n!}.
$$
The Legendre polynomials correspond to $\alpha=\beta=0$.  If
$$
w_{\alpha,\beta}(t)
:=
(1-t)^\alpha(1+t)^\beta,
\quad \alpha,\beta>-1,
\numberthis{eq:jacobi-weight}
$$
and $h_j^{(\alpha,\beta)}$ denotes the squared norm of
$P_j^{(\alpha,\beta)}$ in
$L^2([-1,1],w_{\alpha,\beta}(t)dt)$, then the reproducing
kernel of the $n$th Jacobi partial-sum operator is
$$
K_n^{(\alpha,\beta)}(x,t)
:=
\sum_{j=0}^n
\frac{
P_j^{(\alpha,\beta)}(x)P_j^{(\alpha,\beta)}(t)
}{h_j^{(\alpha,\beta)}}.
\numberthis{eq:jacobi-kernel}
$$
Its Lebesgue function and Lebesgue constant are
$$
\begin{aligned}
\lambda_n^{(\alpha,\beta)}(x)
&:= 
\int_{-1}^1
\bigl|K_n^{(\alpha,\beta)}(x,t)\bigr|
w_{\alpha,\beta}(t)dt,\\[7pt]
\Lambda_n^{(\alpha,\beta)}
&:=
\sup_{-1\leq x\leq1}
\lambda_n^{(\alpha,\beta)}(x).
\end{aligned}
\numberthis{eq:lebesgue-definitions}
$$

Szeg\H{o} conjectured in 1926 that the Legendre Lebesgue constants
$\bigl(\Lambda_n^{(0,0)}\bigr)_{n\geq0}$ form a strictly increasing
sequence.  Qu and Wong
proved the conjecture in 1988 \cite{QuWong}.  Their proof used a
four-term asymptotic expansion with an explicit error bound
for $n>49$, followed by numerical verification of the remaining
degrees.  In the conclusion of their paper they described the method
as ``far too complicated'' and proposed finding ``an alternative
expression for the Lebesgue constants from which the monotonicity of
these constants is evident'' \cite[p.~187]{QuWong}.

The present note is based on an earlier preprint of the first author
\cite{Castillo2026}, which settled one of the most prominent open
problems in the area.  The equal-phase Pr\"ufer framework developed
there is used here to realise the proposal of Qu and Wong.  The
present paper applies that framework to a different problem; it does
not extend the result of \cite{Castillo2026}.

Lebesgue constants for general Jacobi series have a long history,
beginning with Rau \cite{Rau1929} and continuing through the work of
Lorch \cite{LorchI,LorchII} and Frenzen and Wong
\cite{FrenzenWong}.  That literature provides integral
representations and detailed asymptotic information.  To the best of
our knowledge, it does not contain the all-degree comparison proved
below.

We realise that proposal and generalise the Wong--Zhang extremal
comparison \cite{WongZhang} to a
one-parameter Jacobi family.  Fix $\beta>-1$ and put
$$
Q_m^{(\beta)}(x)
:=
2^{-\beta-1}(1+x)^{\beta+1}
P_{m-1}^{(0,\beta+1)}(x),
\quad m\geq1.
\numberthis{eq:Q-definition}
$$
The weighted Jacobi differentiation identity gives
$$
\bigl(Q_m^{(\beta)}\bigr)'(x)
=
\frac{m+\beta}{2^{\beta+1}}
(1+x)^\beta P_{m-1}^{(1,\beta)}(x),
\quad -1<x<1.
\numberthis{eq:Q-derivative}
$$
For $\beta=0$, this specialises to
$$
Q_m^{(0)}(x)
=
\frac{1+x}{2}P_{m-1}^{(0,1)}(x)
=
P_m^{(0,-1)}(x),
\quad m\geq1,
$$
since $(1+x)P_{m-1}^{(0,1)}(x)=2P_m^{(0,-1)}(x)$.
Moreover, the Christoffel--Darboux formula at $x=1$ shows that
$$
(1+x)^\beta K_{m-1}^{(0,\beta)}(1,x)
=
\bigl(Q_m^{(\beta)}\bigr)'(x),
\quad -1<x<1.
\numberthis{eq:kernel-primitive}
$$
Thus the endpoint value $\lambda_n^{(0,\beta)}(1)$ is exactly the
total variation of $Q_{n+1}^{(\beta)}$.

Equation~\eqref{eq:Q-derivative} has $m-1$ simple
zeros in $(-1,1)$.  Enumerate them from right to left as
$$
1>\xi_{1,m}^{(\beta)}>
\xi_{2,m}^{(\beta)}>
\cdots>
\xi_{m-1,m}^{(\beta)}>-1.
$$
Our principal result, Theorem~\ref{thm:equal-phase}, states that, for
$\beta\geq-1/3$, consecutive functions $Q_m^{(\beta)}$ are ordered
not only at these relative
extrema but at every common Pr\"ufer phase.  In particular,
$$
\left|
Q_{m+1}^{(\beta)}\bigl(\xi_{k,m+1}^{(\beta)}\bigr)
\right|
>
\left|
Q_m^{(\beta)}\bigl(\xi_{k,m}^{(\beta)}\bigr)
\right|,
\quad 1\leq k\leq m-1.
$$
The resulting total-variation formula is
$$
\lambda_n^{(0,\beta)}(1)
=
1+2\sum_{k=1}^n
\left|
Q_{n+1}^{(\beta)}\bigl(\xi_{k,n+1}^{(\beta)}\bigr)
\right|.
$$
Every term already present on the right-hand side increases strictly
when the degree is raised, and the next degree contributes one new
positive term.  Hence the endpoint sequence
$(\lambda_n^{(0,\beta)}(1))_{n\geq0}$ is strictly increasing for every
$\beta\geq-1/3$.

For $-1/3\leq\beta\leq0$, Gasper's positive product formula for
Jacobi polynomials \cite{Gasper1971} implies
$$
\Lambda_n^{(0,\beta)}
=
\lambda_n^{(0,\beta)}(1).
$$
We consequently obtain strict monotonicity of the global Lebesgue
constants throughout that parameter interval.  By reflection, the same
conclusion therefore holds for the sequence
$(\Lambda_n^{(\gamma,0)})_{n\geq0}$, where
$-1/3\leq\gamma\leq0$.
This gives an all-degree strict monotonicity theorem throughout a
non-degenerate interval of Jacobi parameters.

At $\beta=0$, write $P_j=P_j^{(0,0)}$ for the Legendre polynomial
of degree $j$.  The contiguous relation
$$
Q_m^{(0)}(x)
=
\frac{1+x}{2}P_{m-1}^{(0,1)}(x)
=
P_m^{(0,-1)}(x)
=
\frac{P_m(x)+P_{m-1}(x)}2
\numberthis{eq:legendre-contiguous}
$$
connects the problem with a conjecture of Askey on the relative
extrema of $P_m^{(0,-1)}$ \cite[Conjecture~1, p.~23]{Askey1990}.
Wong and Zhang proved these discrete inequalities by asymptotic
methods \cite{WongZhang}.  Our continuous theorem gives a direct,
non-asymptotic proof of their result.  Combined with
\eqref{eq:ell-total-variation}, it also gives a short proof of the
Qu--Wong theorem.  It is the representation and the termwise
comparison together, rather than the representation in isolation,
that make the monotonicity immediate and realise the approach
proposed by Qu and Wong.

The lower bound $-1/3$ is sharp for the continuous comparison: when
$-1<\beta<-1/3$, the amplitude ordering is reversed for all
sufficiently small positive phases.  This local reversal does not
decide the optimal range of the discrete inequalities, whose phases
are $k\pi$.  Nor do we claim a universal monotonicity theorem in the
full Jacobi parameter plane.  The cumulative endpoint kernel satisfies
the homogeneous constant-frequency second-order equation required by
our Pr\"ufer argument precisely when $\alpha=0$;
Section~\ref{sec:general} identifies the obstructing term for
$\alpha\ne0$.

The proof adapts the equal-phase method developed in
\cite{Castillo2026}.  Section~\ref{sec:radau} derives the exact
Jacobi--Radau identities.  Section~\ref{sec:phase} constructs the
Pr\"ufer phase and determines its range.  Section~\ref{sec:comparison}
proves the continuous comparison, including the boundary case
$\beta=-1/3$.  Section~\ref{sec:lebesgue} derives the endpoint and
global Lebesgue results, and Section~\ref{sec:general} explains the
scope of the method in the two-parameter family.

\section{The weighted Jacobi--Radau family}
\label{sec:radau}

Unless explicitly stated otherwise, throughout
Sections~\ref{sec:radau}--\ref{sec:comparison} we fix
$$
\beta\geq-\frac13.
$$
We suppress the superscript $(\beta)$ from $Q_m^{(\beta)}$ and from
the phase variables when no ambiguity can arise.  Put
$$
\omega_m:=\sqrt{m(m+\beta)}.
\numberthis{eq:omega-definition}
$$

For $\alpha=0$, the Jacobi norm has the simple form
$$
h_j^{(0,\beta)}
=
\int_{-1}^1
\bigl(P_j^{(0,\beta)}(t)\bigr)^2(1+t)^\beta dt
=
\frac{2^{\beta+1}}{2j+\beta+1}.
$$
Since $P_j^{(0,\beta)}(1)=1$, the endpoint kernel is
$$
K_n^{(0,\beta)}(1,x)
=
\frac1{2^{\beta+1}}
\sum_{j=0}^n
(2j+\beta+1)P_j^{(0,\beta)}(x).
$$
The endpoint form of the Christoffel--Darboux identity is
$$
K_n^{(0,\beta)}(1,x)
=
\frac{n+\beta+1}{2^{\beta+1}}
P_n^{(1,\beta)}(x).
\numberthis{eq:endpoint-kernel}
$$
These standard identities may be found, for example, in
\cite[Chapter~IV]{Szego}.

The weighted differentiation formula reads
$$
\frac{d}{dx}
\left\{
(1+x)^{\beta+1}P_{m-1}^{(0,\beta+1)}(x)
\right\}
=
(m+\beta)(1+x)^\beta P_{m-1}^{(1,\beta)}(x),
\quad -1<x<1,
\numberthis{eq:weighted-differentiation}
$$
Combining \eqref{eq:Q-definition},
\eqref{eq:endpoint-kernel}, and
\eqref{eq:weighted-differentiation} gives
\eqref{eq:Q-derivative} and \eqref{eq:kernel-primitive}.  It also
shows that $Q_m$ is absolutely continuous on $[-1,1]$, including
when $\beta<0$.  Since
$P_{m-1}^{(0,\beta+1)}(1)=1$ and
$P_{m-1}^{(1,\beta)}(1)=m$, equations
\eqref{eq:Q-definition} and \eqref{eq:Q-derivative} give
$$
Q_m(1)=1,
\quad
Q_m(-1)=0,
\quad
Q_m'(1)=\frac{\omega_m^2}{2}.
\numberthis{eq:Q-endpoints}
$$

For clarity, the substitution leading to the differential equation
is recorded explicitly.  Put $p=\beta+1$ and
$U=P_{m-1}^{(0,p)}$.  By \eqref{eq:Q-definition},
$$
\begin{aligned}
U
&=
2^p(1+x)^{-p}Q_m,\\[7pt]
U'
&=
2^p(1+x)^{-p}
\left(
Q_m'-\frac{p}{1+x}Q_m
\right),\\[7pt]
U''
&=
2^p(1+x)^{-p}
\left(
Q_m''-\frac{2p}{1+x}Q_m'
+\frac{p(p+1)}{(1+x)^2}Q_m
\right).
\end{aligned}
\numberthis{eq:radau-substitution}
$$
The Jacobi equation for $U$ is
$$
(1-x^2)U''+\bigl[p-(p+2)x\bigr]U'
+(m-1)(m+p)U=0.
\numberthis{eq:U-jacobi-equation}
$$
Substituting \eqref{eq:radau-substitution} into
\eqref{eq:U-jacobi-equation} and cancelling
$2^p(1+x)^{-p}$ gives a coefficient
$$
-2p(1-x)+p-(p+2)x
=
(p-2)x-p
$$
for $Q_m'$, while the coefficient of $Q_m$ is
$$
p+(m-1)(m+p)
=
m(m+p-1).
$$
Since $p=\beta+1$, one obtains the Jacobi--Radau equation
$$
(1-x^2)Q_m''(x)
+\bigl[(\beta-1)x-(\beta+1)\bigr]Q_m'(x)
+\omega_m^2Q_m(x)
=0,
\quad -1<x<1.
\numberthis{eq:Q-equation}
$$
This equation, rather than a large-degree approximation, is the
starting point of the comparison.

Set
$$
Y_m(\theta):=Q_m(\cos\theta),
\quad 0\leq\theta\leq\pi.
$$
The change of variables $x=\cos\theta$ gives
$$
Y_m'(\theta)
=
-Q_m'(\cos\theta)\sin\theta,
\quad
Y_m''(\theta)
=
-Q_m'(\cos\theta)\cos\theta
+Q_m''(\cos\theta)\sin^2\theta.
\numberthis{eq:angular-derivatives}
$$
Substitution of \eqref{eq:angular-derivatives} into
\eqref{eq:Q-equation} gives the angular equation
$$
Y_m''(\theta)
+H_\beta(\theta)Y_m'(\theta)
+\omega_m^2Y_m(\theta)
=0,
\quad 0<\theta<\pi,
\numberthis{eq:angular-equation}
$$
where
$$
H_\beta(\theta)
:=
\frac{\beta+1-\beta\cos\theta}{\sin\theta}.
\numberthis{eq:H-definition}
$$

Equation~\eqref{eq:Q-endpoints} and the Taylor expansion of the
cosine give, at the left endpoint,
$$
Y_m(\theta)
=
1-\frac{\omega_m^2}{4}\theta^2+O(\theta^4),
\quad
Y_m'(\theta)
=
-\frac{\omega_m^2}{2}\theta+O(\theta^3)
\numberthis{eq:left-expansion}
$$
as $\theta\downarrow0$.  At the right endpoint, put
$$
\rho:=2(\beta+1),
\quad
C_{m,\beta}
:=
(-1)^{m-1}
\frac{(\beta+2)_{m-1}}{(m-1)!2^\rho}.
\numberthis{eq:right-constants}
$$
Since
$P_{m-1}^{(0,\beta+1)}(-1)
=(-1)^{m-1}(\beta+2)_{m-1}/(m-1)!$, one has
$$
\begin{aligned}
Y_m(\pi-\varepsilon)
&=
C_{m,\beta}\varepsilon^\rho
\bigl(1+O(\varepsilon^2)\bigr),
\\[7pt]
Y_m'(\pi-\varepsilon)
&=
-\rho C_{m,\beta}\varepsilon^{\rho-1}
\bigl(1+O(\varepsilon^2)\bigr)
\end{aligned}
\numberthis{eq:right-expansion}
$$
as $\varepsilon\downarrow0$.

\section{The Pr\"ufer phase and its exact range}
\label{sec:phase}

For $0\leq\theta<\pi$, the vector
$$
\left(Y_m(\theta),-\frac{Y_m'(\theta)}{\omega_m}\right)
$$
does not vanish.  At an interior point this follows from uniqueness
for \eqref{eq:angular-equation}, and at $\theta=0$ it follows from
$Y_m(0)=1$.  Define
$$
r_m(\theta)
:=
\left(
Y_m(\theta)^2+\frac{Y_m'(\theta)^2}{\omega_m^2}
\right)^{1/2}.
\numberthis{eq:radius-definition}
$$
Then $r_m>0$ on $[0,\pi)$, while
\eqref{eq:right-constants}--\eqref{eq:right-expansion} and $\rho>1$
show that $r_m(\pi)=0$ by continuity.

The corresponding path in the unit circle has a unique continuous
lift $\phi_m\colon[0,\pi)\to\mathbb R$ satisfying
$\phi_m(0)=0$.  Thus one has the Pr\"ufer representation
$$
Y_m=r_m\cos\phi_m,
\quad
Y_m'=-\omega_m r_m\sin\phi_m.
\numberthis{eq:prufer}
$$
Differentiating \eqref{eq:prufer} and substituting
\eqref{eq:angular-equation} gives the phase and amplitude equations
$$
\begin{aligned}
\phi_m'
&=
\omega_m-H_\beta\sin\phi_m\cos\phi_m,
\\[7pt]
\frac{r_m'}{r_m}
&=
-H_\beta\sin^2\phi_m.
\end{aligned}
\numberthis{eq:phase-amplitude}
$$

\begin{lemma}\label{lem:phase-monotonicity}
For every $m\geq1$,
$$
\phi_m'(\theta)>0,
\quad 0<\theta<\pi.
$$
\end{lemma}

\begin{proof}
Put
$$
G_m(\theta)
:=
\frac{\phi_m'(\theta)}{\omega_m}
=
1-
\frac{H_\beta(\theta)}{\omega_m}
\sin\phi_m(\theta)\cos\phi_m(\theta).
\numberthis{eq:G-definition}
$$
Equations \eqref{eq:prufer} and \eqref{eq:radius-definition} give
the identity
$$
\sin\phi_m\cos\phi_m
=
-\frac{\omega_mY_mY_m'}
{\omega_m^2Y_m^2+(Y_m')^2}
\numberthis{eq:sin-cos-identity}
$$
and \eqref{eq:left-expansion} give
$$
\sin\phi_m(\theta)\cos\phi_m(\theta)
=
\frac{\omega_m}{2}\theta+O(\theta^3).
$$
Equation \eqref{eq:H-definition} gives
$H_\beta(\theta)=\theta^{-1}+O(\theta)$; hence
\eqref{eq:G-definition} yields
$$
G_m(0+)=\frac12.
\numberthis{eq:G-left}
$$

At the other endpoint, \eqref{eq:right-expansion} and
\eqref{eq:sin-cos-identity} give
$$
\sin\phi_m(\pi-\varepsilon)
\cos\phi_m(\pi-\varepsilon)
=
\frac{\omega_m}{\rho}\varepsilon+O(\varepsilon^3).
$$
Moreover, \eqref{eq:H-definition} gives
$$
H_\beta(\pi-\varepsilon)
=
\frac{2\beta+1}{\varepsilon}+O(\varepsilon).
$$
Since $\rho=2(\beta+1)$, equation \eqref{eq:G-definition} yields
$$
G_m(\pi-)
=
1-\frac{2\beta+1}{\rho}
=
\frac{1}{2(\beta+1)}>0.
\numberthis{eq:G-right}
$$

It remains to exclude an interior zero.  If
$G_m(\theta_0)=0$, then \eqref{eq:G-definition} and
\eqref{eq:phase-amplitude} give
$\phi_m'(\theta_0)=0$ and
$H_\beta(\theta_0)\sin\phi_m(\theta_0)
\cos\phi_m(\theta_0)=\omega_m$.  Moreover,
$$
G_m'
=
-\frac{H_\beta'}{\omega_m}\sin\phi_m\cos\phi_m
-\frac{H_\beta}{\omega_m}\cos(2\phi_m)\phi_m'.
$$
Consequently,
$$
G_m'(\theta_0)
=
-\frac{H_\beta'(\theta_0)}{H_\beta(\theta_0)}.
\numberthis{eq:G-zero-derivative}
$$
To use this identity, set
$$
a:=2\beta+1,
\quad
t:=\tan\frac\theta2.
$$
Then $a\geq1/3$, and \eqref{eq:H-definition} gives
$$
H_\beta(\theta)
=
\frac{1+at^2}{2t}.
\numberthis{eq:H-half-angle}
$$
Thus $H_\beta$ decreases until
$t_*=a^{-1/2}$ and increases thereafter, with minimum
$H_\beta(2\arctan t_*)=\sqrt a$.  Since
$|\sin\phi_m\cos\phi_m|\leq1/2$, equations
\eqref{eq:G-definition} and \eqref{eq:H-half-angle} give, at that
minimum,
$$
G_m
\geq
1-\frac{\sqrt a}{2\omega_m}>0,
$$
because
$$
4\omega_m^2-a
=
4m(m+\beta)-(2\beta+1)
\geq
2\beta+3>0.
$$

If $G_m$ had a zero before the minimum of $H_\beta$, choose its
first zero.  By \eqref{eq:G-left}, the derivative there would be
non-positive, whereas \eqref{eq:G-zero-derivative} and $H_\beta'<0$
would make it positive.  If $G_m$ had a zero after the minimum,
choose its last zero, using \eqref{eq:G-right}; its derivative would
be non-negative, whereas \eqref{eq:G-zero-derivative} and
$H_\beta'>0$ would make it negative.  A zero at the minimum has
already been excluded.  Hence $G_m>0$ throughout $(0,\pi)$.
\end{proof}

The vanishing of $r_m$ at $\pi$ means that the terminal phase cannot
be read directly from the Pr\"ufer representation.  The following
count fixes the correct lift.

\begin{lemma}\label{lem:phase-range}
Let
$$
S_m:=\left(m-\frac12\right)\pi.
\numberthis{eq:S-definition}
$$
Then
$$
\lim_{\theta\uparrow\pi}\phi_m(\theta)=S_m.
$$
After setting $\phi_m(\pi):=S_m$, the phase is a homeomorphism
$$
\phi_m\colon[0,\pi]\longrightarrow[0,S_m].
$$
If $\Theta_m:=\phi_m^{-1}$, then
$$
\xi_{k,m}^{(\beta)}
=
\cos\Theta_m(k\pi),
\quad
\left|Q_m\bigl(\xi_{k,m}^{(\beta)}\bigr)\right|
=
r_m\bigl(\Theta_m(k\pi)\bigr),
\numberthis{eq:critical-phase}
$$
for $1\leq k\leq m-1$.  Every
$\xi_{k,m}^{(\beta)}$ is a point at which $|Q_m|$ has a strict local
maximum.
\end{lemma}

\begin{proof}
The factorisation \eqref{eq:Q-definition} shows that $Q_m$ has
exactly $m-1$ simple zeros in $(-1,1)$.  Hence $Y_m$ has exactly
$m-1$ zeros in $(0,\pi)$.  On the other hand,
\eqref{eq:radius-definition} and \eqref{eq:right-expansion} give
$$
r_m(\pi-\varepsilon)
=
\frac{\rho|C_{m,\beta}|}{\omega_m}
\varepsilon^{\rho-1}
\bigl(1+O(\varepsilon^2)\bigr),
\numberthis{eq:radius-right}
$$
and combining \eqref{eq:right-expansion} with
\eqref{eq:radius-right} gives
$$
\frac1{r_m(\pi-\varepsilon)}
\left(
Y_m(\pi-\varepsilon),
-\frac{Y_m'(\pi-\varepsilon)}{\omega_m}
\right)
\longrightarrow
\bigl(0,(-1)^{m-1}\bigr).
\numberthis{eq:terminal-vector}
$$

By Lemma~\ref{lem:phase-monotonicity}, the phase is strictly
increasing.  If it were unbounded as $\theta\uparrow\pi$, it would
cross infinitely many levels at which $\cos\phi_m=0$, contrary to the
finite zero count for $Y_m$.  Thus $\phi_m$ has a finite terminal
limit $L_m$.  Put $\alpha_k=(k+1/2)\pi$ for $k\geq0$.  Since
$\phi_m(0)=0$ and $\phi_m$ is strictly increasing, the $m-1$
interior zeros of $Y_m=r_m\cos\phi_m$ occur at the successive levels
$\alpha_0,\ldots,\alpha_{m-2}$, with this list empty when $m=1$.
There can be no interior crossing of $\alpha_{m-1}$, and hence
$L_m\leq\alpha_{m-1}$.  On the other hand,
\eqref{eq:terminal-vector} gives
$L_m\equiv\alpha_{m-1}\pmod{2\pi}$.  Together with strict increase
and the preceding $m-1$ crossings, this forces
$L_m=\alpha_{m-1}=S_m$.  The homeomorphism assertion follows.

For $0<\theta<\pi$, the Pr\"ufer representation shows that
$Y_m'(\theta)=0$
exactly when $\phi_m(\theta)=k\pi$.  The first identity in
\eqref{eq:angular-derivatives} therefore shows that these levels give
precisely the critical-phase identities stated in the lemma.  At such
a point $\theta_0$,
\eqref{eq:angular-equation} gives
$Y_m''(\theta_0)=-\omega_m^2Y_m(\theta_0)$; hence
$$
\left.
\frac{d^2}{d\theta^2}Y_m(\theta)^2
\right|_{\theta=\theta_0}
=
-2\omega_m^2Y_m(\theta_0)^2<0.
$$
Thus $|Q_m|$ has a strict local maximum at each critical point.
\end{proof}

For completeness, \eqref{eq:prufer} gives
$\cot\phi_m=-\omega_mY_m/Y_m'$.  Hence
\eqref{eq:right-expansion} yields
$$
\cot\phi_m(\pi-\varepsilon)
=
\frac{\omega_m}{\rho}\varepsilon+O(\varepsilon^3).
$$
Writing $\delta_m(\varepsilon)
:=S_m-\phi_m(\pi-\varepsilon)$, one has
$\cot(S_m-\delta_m)=\tan\delta_m$.  Therefore
$$
\phi_m(\pi-\varepsilon)
=
S_m-\frac{\omega_m}{\rho}\varepsilon
+O(\varepsilon^3).
$$

\section{Comparison at equal phase}
\label{sec:comparison}

The proof of Theorem~\ref{thm:equal-phase} adapts the equal-phase
architecture developed in \cite[Section~2]{Castillo2026}.  First, each
Pr\"ufer phase is shown to be strictly increasing; consecutive phases,
and hence their inverses, are then ordered.  A small-phase expansion
establishes the initial order of
$q_j(s)=H_\beta(\Theta_j(s))/\omega_j$, and a first-contact argument
based on $\mathcal L_\beta=-H_\beta'/H_\beta^2$ propagates that order
across the common phase interval.  Finally, integration of the
logarithmic derivative of $R_{m+1}/R_m$ yields the amplitude
comparison.

\subsection{Ordering the phases}

\begin{lemma}\label{lem:phase-comparison}
For every $m\geq1$,
$$
\phi_{m+1}(\theta)>\phi_m(\theta),
\quad 0<\theta\leq\pi.
\numberthis{eq:phase-ordering}
$$
Consequently,
$$
\Theta_{m+1}(s)<\Theta_m(s),
\quad 0<s\leq S_m.
\numberthis{eq:inverse-ordering}
$$
\end{lemma}

\begin{proof}
Equations \eqref{eq:left-expansion} and \eqref{eq:prufer} give
$$
\phi_j(\theta)
=
\frac{\omega_j}{2}\theta+O(\theta^3).
\numberthis{eq:phase-left-first}
$$
By \eqref{eq:omega-definition} and \eqref{eq:phase-left-first},
$\omega_{m+1}^2-\omega_m^2=2m+\beta+1>0$, so the difference
$\phi_{m+1}-\phi_m$ is positive near zero.  If it had a first zero
$\theta_0$ in $(0,\pi)$, its derivative there would be non-positive.
At a point where the phases agree, however,
the first identity in \eqref{eq:phase-amplitude} gives
$$
\bigl(\phi_{m+1}-\phi_m\bigr)'(\theta_0)
=
\omega_{m+1}-\omega_m>0,
$$
a contradiction.  The inequality at $\pi$ follows from
$S_{m+1}>S_m$.

For $0<s<S_m$, apply \eqref{eq:phase-ordering} at
$\Theta_m(s)$ and use the strict increase of $\phi_{m+1}$.  At
$s=S_m$, use \eqref{eq:S-definition},
$\Theta_m(S_m)=\pi$, and
$\phi_{m+1}(\pi)=S_{m+1}>S_m$.
\end{proof}

\subsection{The normalised damping coefficient}

For $0<s<S_j$, define
$$
q_j(s)
:=
\frac{H_\beta(\Theta_j(s))}{\omega_j},
\quad
D_j(s)
:=
1-q_j(s)\sin s\cos s.
\numberthis{eq:qD-definition}
$$
The first identity in \eqref{eq:phase-amplitude} and
Lemma~\ref{lem:phase-monotonicity} give the positive phase-speed
identity
$$
D_j(s)
=
\frac{\phi_j'(\Theta_j(s))}{\omega_j}>0.
\numberthis{eq:D-positive}
$$

\begin{lemma}\label{lem:q-local}
For every $m\geq1$, there exists $\delta>0$ such that
$$
q_{m+1}(s)<q_m(s),
\quad 0<s<\delta.
$$
\end{lemma}

\begin{proof}
Expansion of \eqref{eq:H-definition} at $\theta=0$ gives
$$
H_\beta(\theta)
=
\frac1\theta+\frac{3\beta+1}{6}\theta+O(\theta^3).
\numberthis{eq:H-left-expansion}
$$
The function $Y_j$ is even and analytic near zero and is positive
there, so the local identity
$$
\phi_j(\theta)
=
\arctan\left(-\frac{Y_j'(\theta)}{\omega_jY_j(\theta)}\right)
$$
shows that the phase is odd and analytic.  We shall use
$$
\sin\phi\cos\phi
=
\phi-\frac23\phi^3+O(\phi^5)
\numberthis{eq:local-trig-expansion}
$$
Substitution of \eqref{eq:H-left-expansion} and
\eqref{eq:local-trig-expansion} into the first identity in
\eqref{eq:phase-amplitude} makes the coefficient calculation
explicit.  Put
$h_1=(3\beta+1)/6$ and
$\phi_j(\theta)=a_1\theta+a_3\theta^3+O(\theta^5)$.  Comparing the
constant and $\theta^2$ coefficients gives
$$
2a_1=\omega_j,
\quad
4a_3=\frac23a_1^3-h_1a_1.
$$
Thus
$$
\phi_j(\theta)
=
\frac{\omega_j}{2}\theta
+\frac{\omega_j(\omega_j^2-3\beta-1)}{48}\theta^3
+O(\theta^5).
\numberthis{eq:phase-left-expansion}
$$
Reversion of \eqref{eq:phase-left-expansion} yields
$$
\Theta_j(s)
=
\frac{2s}{\omega_j}
-\frac{\omega_j^2-3\beta-1}{3\omega_j^3}s^3
+O(s^5).
\numberthis{eq:inverse-left-expansion}
$$
Substituting \eqref{eq:inverse-left-expansion} into
\eqref{eq:H-left-expansion} and then using
\eqref{eq:qD-definition} gives
$$
q_j(s)
=
\frac1{2s}
+\left(
\frac1{12}+\frac{3\beta+1}{4\omega_j^2}
\right)s
+O(s^3).
\numberthis{eq:q-left-expansion}
$$
If $\beta>-1/3$, it follows that
$$
q_{m+1}(s)-q_m(s)
=
\frac{3\beta+1}{4}
\left(
\frac1{\omega_{m+1}^2}-\frac1{\omega_m^2}
\right)s
+O(s^3)
\numberthis{eq:q-local-difference}
$$
and this is negative for sufficiently small $s>0$.

At the boundary $\beta=-1/3$, the leading coefficient in
\eqref{eq:q-local-difference} vanishes and one more term is
necessary.  Expansion of \eqref{eq:H-definition} and the Taylor
formula for $\sin\phi\cos\phi$ give, respectively,
$$
H_{-1/3}(\theta)
=
\frac1\theta+\frac1{180}\theta^3+O(\theta^5),
\numberthis{eq:boundary-H-expansion}
$$
and
$$
\sin\phi\cos\phi
=
\phi-\frac23\phi^3+\frac{2}{15}\phi^5+O(\phi^7).
\numberthis{eq:boundary-trig-expansion}
$$
Writing
$\phi_j(\theta)=a_1\theta+a_3\theta^3+a_5\theta^5+O(\theta^7)$
and substituting \eqref{eq:boundary-H-expansion} and
\eqref{eq:boundary-trig-expansion} in the first identity of
\eqref{eq:phase-amplitude} gives
$$
2a_1=\omega_j,
\quad
4a_3=\frac23a_1^3,
\quad
6a_5
=
2a_1^2a_3-\frac{2}{15}a_1^5-\frac{a_1}{180}.
$$
Consequently,
$$
\phi_j(\theta)
=
\frac{\omega_j}{2}\theta
+\frac{\omega_j^3}{48}\theta^3
+\left(
\frac{\omega_j^5}{960}-\frac{\omega_j}{2160}
\right)\theta^5
+O(\theta^7).
\numberthis{eq:boundary-phase-expansion}
$$
Reversion of \eqref{eq:boundary-phase-expansion} gives
$$
\Theta_j(s)
=
\frac{2s}{\omega_j}
-\frac{s^3}{3\omega_j}
+\left(
\frac1{10\omega_j}+\frac4{135\omega_j^5}
\right)s^5
+O(s^7).
\numberthis{eq:boundary-inverse-expansion}
$$
Substitution of \eqref{eq:boundary-inverse-expansion} in
\eqref{eq:boundary-H-expansion}, followed by
\eqref{eq:qD-definition}, gives
$$
q_j(s)
=
\frac1{2s}+\frac{s}{12}
+\left(
-\frac1{90}+\frac1{27\omega_j^4}
\right)s^3
+O(s^5).
\numberthis{eq:boundary-q-expansion}
$$
Subtracting \eqref{eq:boundary-q-expansion} at $j=m+1$ and $j=m$
gives
$$
q_{m+1}(s)-q_m(s)
=
\frac1{27}
\left(
\frac1{\omega_{m+1}^4}-\frac1{\omega_m^4}
\right)s^3
+O(s^5)
\numberthis{eq:q-boundary-difference}
$$
Since $\omega_{m+1}>\omega_m$, \eqref{eq:q-boundary-difference} is
negative for sufficiently small $s>0$.
\end{proof}

The local order propagates through the entire common phase range.

\begin{lemma}\label{lem:q-comparison}
For every $m\geq1$,
$$
q_{m+1}(s)<q_m(s),
\quad 0<s<S_m.
\numberthis{eq:q-comparison}
$$
\end{lemma}

\begin{proof}
The inverse-function theorem and \eqref{eq:D-positive} give
$$
\Theta_j'(s)
=
\frac1{\omega_jD_j(s)}.
\numberthis{eq:inverse-phase-derivative}
$$
Define
$$
\mathcal L_\beta(\theta)
:=
-\frac{H_\beta'(\theta)}{H_\beta(\theta)^2}.
\numberthis{eq:L-definition}
$$
Differentiating the first definition in \eqref{eq:qD-definition} and
using \eqref{eq:inverse-phase-derivative} and
\eqref{eq:L-definition} gives
$$
q_j'(s)
=
-\frac{
\mathcal L_\beta(\Theta_j(s))q_j(s)^2
}{D_j(s)}.
\numberthis{eq:q-equation}
$$

In terms of $a=2\beta+1$ and $t=\tan(\theta/2)$,
$$
\widehat{\mathcal L}_\beta(t)
:=
\mathcal L_\beta(2\arctan t)
=
\frac{(1-at^2)(1+t^2)}{(1+at^2)^2}.
\numberthis{eq:L-half-angle}
$$
Moreover,
$$
\frac{d}{dt}\widehat{\mathcal L}_\beta(t)
=
\frac{
2t\{a(a-3)t^2+1-3a\}
}{(1+at^2)^3}<0,
\quad 0<t<a^{-1/2}.
\numberthis{eq:L-derivative}
$$
Indeed, if $1/3\leq a\leq3$, the expression in braces is negative
for every $t>0$.  If $a>3$, its supremum on
$0<t<a^{-1/2}$ is bounded above by
$$
(a-3)+1-3a=-2a-2<0.
$$

Set $f=q_{m+1}-q_m$.  By Lemma~\ref{lem:q-local}, $f<0$ near
zero.  Suppose that $f$ has a first zero $s_0$ in $(0,S_m)$ and
write
$$
q_{m+1}(s_0)=q_m(s_0)=:q>0,
$$
where positivity follows from $H_\beta>0$ on $(0,\pi)$.
At this point $D_{m+1}(s_0)=D_m(s_0)=:D>0$.  Put
$$
t_+
:=
\tan\frac{\Theta_{m+1}(s_0)}2,
\quad
t_0
:=
\tan\frac{\Theta_m(s_0)}2.
$$
By Lemma~\ref{lem:phase-comparison}, $t_+<t_0$, while equality of
the two $q$'s gives
$$
H_\beta(2\arctan t_+)
=
q\omega_{m+1}
>
q\omega_m
=
H_\beta(2\arctan t_0).
\numberthis{eq:H-crossing}
$$
If $t_+\geq a^{-1/2}$, then $t_+<t_0$ places both points on the
strictly increasing branch of \eqref{eq:H-half-angle}, contrary to
\eqref{eq:H-crossing}.  Hence $t_+<a^{-1/2}$.  If
$t_0\geq a^{-1/2}$, formula \eqref{eq:L-half-angle} gives
$$
\widehat{\mathcal L}_\beta(t_+)
>0\geq\widehat{\mathcal L}_\beta(t_0).
$$
If $t_0<a^{-1/2}$, the same strict inequality follows from the strict
decrease in \eqref{eq:L-derivative}.  Hence \eqref{eq:q-equation}
gives
$$
f'(s_0)
=
-\frac{q^2}{D}
\left\{
\widehat{\mathcal L}_\beta(t_+)
-\widehat{\mathcal L}_\beta(t_0)
\right\}<0.
$$
This is incompatible with a first crossing from negative values,
which would require $f'(s_0)\geq0$.  Thus no first zero exists, and
the asserted global ordering of $q_{m+1}$ and $q_m$ follows.
\end{proof}

\subsection{The continuous comparison}

For $0\leq s\leq S_j$, put
$$
R_j(s)
:=
r_j\bigl(\Theta_j(s)\bigr).
\numberthis{eq:R-definition}
$$
Thus $R_j(0)=1$ and $R_j(S_j)=0$.

\begin{theorem}[Continuous equal-phase comparison]
\label{thm:equal-phase}
Let $\beta\geq-1/3$.  For every $m\geq1$,
$$
\Theta_{m+1}(s)<\Theta_m(s)
\quad\hbox{and}\quad
R_{m+1}(s)>R_m(s),
\quad 0<s\leq S_m.
$$
\end{theorem}

\begin{proof}
The first inequality is \eqref{eq:inverse-ordering}.  For
$0<s<S_m$, the chain rule, the second identity in
\eqref{eq:phase-amplitude}, \eqref{eq:qD-definition}, and
\eqref{eq:inverse-phase-derivative} give
$$
\frac{d}{ds}\log R_j(s)
=
-\frac{q_j(s)\sin^2s}{D_j(s)}.
\numberthis{eq:R-log-derivative}
$$
By \eqref{eq:qD-definition}, with $u=\sin s\cos s$, the mixed terms
cancel exactly:
$$
\begin{aligned}
q_mD_{m+1}-q_{m+1}D_m
&=
q_m(1-q_{m+1}u)-q_{m+1}(1-q_mu)\\[7pt]
&=
q_m-q_{m+1}.
\end{aligned}
\numberthis{eq:amplitude-cancellation}
$$
Subtracting \eqref{eq:R-log-derivative} for the two consecutive
indices and using \eqref{eq:amplitude-cancellation} gives
$$
\frac{d}{ds}
\log\frac{R_{m+1}(s)}{R_m(s)}
=
\frac{
\bigl(q_m(s)-q_{m+1}(s)\bigr)\sin^2s
}{D_m(s)D_{m+1}(s)}.
\numberthis{eq:log-ratio}
$$
By \eqref{eq:q-comparison} and \eqref{eq:D-positive}, the right-hand
side is non-negative and is strictly positive away from the discrete
set $\pi\mathbb Z$.  It remains only to verify the initial value used
in the integration.  Equations \eqref{eq:q-left-expansion} and
\eqref{eq:qD-definition} give
$$
D_j(s)=\frac12+O(s^2).
$$
Equation \eqref{eq:R-log-derivative} therefore gives
$$
\frac{d}{ds}\log R_j(s)=-s+O(s^3),
\quad
R_j(s)=1+O(s^2).
$$
In particular, $R_{m+1}(s)/R_m(s)\to1$ as $s\downarrow0$.
Integration of \eqref{eq:log-ratio} over
$[\varepsilon,s]$, followed by passage to the limit
$\varepsilon\downarrow0$, gives
$$
R_{m+1}(s)>R_m(s),
\quad 0<s<S_m.
$$
At $s=S_m$, \eqref{eq:R-definition} gives
$R_m(S_m)=r_m(\pi)=0$, whereas \eqref{eq:inverse-ordering} gives
$\Theta_{m+1}(S_m)<\pi$ and hence $R_{m+1}(S_m)>0$.
\end{proof}

\begin{corollary}\label{cor:extrema}
Let $\beta\geq-1/3$.  For every $m\geq2$ and
$1\leq k\leq m-1$,
$$
\begin{aligned}
\xi_{k,m+1}^{(\beta)}
&>
\xi_{k,m}^{(\beta)},
\\[7pt]
\left|
Q_{m+1}^{(\beta)}\bigl(\xi_{k,m+1}^{(\beta)}\bigr)
\right|
&>
\left|
Q_m^{(\beta)}\bigl(\xi_{k,m}^{(\beta)}\bigr)
\right|.
\end{aligned}
$$
\end{corollary}

\begin{proof}
Take $s=k\pi$ in Theorem~\ref{thm:equal-phase} and use the
critical-phase identities \eqref{eq:critical-phase}.  The
position inequality follows because
the cosine is strictly decreasing on $(0,\pi)$.
\end{proof}

The parameter endpoint in Theorem~\ref{thm:equal-phase} is sharp for
the continuous statement.

\begin{proposition}[Sharpness at small phase]
\label{prop:sharpness}
Let $-1<\beta<-1/3$ and $m\geq1$.  For $j=m,m+1$, put
$$
Y_j(\theta):=Q_j^{(\beta)}(\cos\theta),
\quad
r_j(\theta):=
\left(
Y_j(\theta)^2+\frac{Y_j'(\theta)^2}{\omega_j^2}
\right)^{1/2},
$$
and choose the local phase $\phi_j$ by
$$
Y_j=r_j\cos\phi_j,
\quad
Y_j'=-\omega_jr_j\sin\phi_j,
\quad
\phi_j(0)=0.
$$
Let $\Theta_j$ denote the local inverse of $\phi_j$ and set
$R_j=r_j\circ\Theta_j$.  Then
$$
R_{m+1}(s)-R_m(s)
=
\frac{3\beta+1}{4}
\left(
\frac1{\omega_m^2}-\frac1{\omega_{m+1}^2}
\right)s^4
+O(s^6).
$$
Consequently,
$R_{m+1}(s)<R_m(s)$ for all sufficiently small $s>0$.
\end{proposition}

\begin{proof}
For any fixed $\beta>-1$, the phase has a strictly increasing local
branch at the origin.  The calculation leading to
\eqref{eq:q-left-expansion} uses only $\beta>-1$ and therefore remains
valid there.  Likewise, the algebraic identity
\eqref{eq:log-ratio} is local and remains valid for this branch.
Substitution of \eqref{eq:q-left-expansion} into
\eqref{eq:log-ratio}, using \eqref{eq:qD-definition}, gives
$D_j(s)=1/2+O(s^2)$ and hence
$$
\frac{d}{ds}
\log\frac{R_{m+1}(s)}{R_m(s)}
=
(3\beta+1)
\left(
\frac1{\omega_m^2}-\frac1{\omega_{m+1}^2}
\right)s^3
+O(s^5).
\numberthis{eq:sharp-log-derivative}
$$
Integration of \eqref{eq:sharp-log-derivative} gives
$$
\log\frac{R_{m+1}(s)}{R_m(s)}
=
\frac{3\beta+1}{4}
\left(
\frac1{\omega_m^2}-\frac1{\omega_{m+1}^2}
\right)s^4
+O(s^6).
\numberthis{eq:sharp-log-ratio}
$$
Moreover, \eqref{eq:R-log-derivative} gives
$$
\frac{d}{ds}\log R_m(s)=-s+O(s^3),
$$
so $R_m(s)=1+O(s^2)$.  Exponentiating
\eqref{eq:sharp-log-ratio} and multiplying by this expansion of
$R_m$ gives the asserted small-phase expansion.  Its leading
coefficient is negative when $\beta<-1/3$.
\end{proof}

\begin{remark}
Proposition~\ref{prop:sharpness} concerns the continuous comparison
near $s=0$.  The relative extrema occur only at the separated phases
$s=k\pi$, so the proposition neither disproves the discrete
inequalities below $-1/3$ nor determines their optimal parameter
range.
\end{remark}

\section{Lebesgue functions and constants}
\label{sec:lebesgue}

We now restore the superscript $(\beta)$ on $Q_m^{(\beta)}$.  Define
the endpoint quantity
$$
\ell_n^{(\beta)}
:=
\lambda_n^{(0,\beta)}(1)
=
\int_{-1}^1
\bigl|K_n^{(0,\beta)}(1,t)\bigr|(1+t)^\beta dt.
\numberthis{eq:ell-definition}
$$
By \eqref{eq:kernel-primitive},
$$
\ell_n^{(\beta)}
=
\int_{-1}^1
\left|\bigl(Q_{n+1}^{(\beta)}\bigr)'(t)\right|dt.
\numberthis{eq:ell-variation}
$$
Thus $\ell_n^{(\beta)}$ is the total variation of
$Q_{n+1}^{(\beta)}$.

\begin{lemma}[Exact total-variation formula]
\label{lem:total-variation}
For $\beta\geq-1/3$ and $n\geq0$,
$$
\ell_n^{(\beta)}
=
1+2\sum_{k=1}^n
\left|
Q_{n+1}^{(\beta)}
\bigl(\xi_{k,n+1}^{(\beta)}\bigr)
\right|,
\numberthis{eq:ell-total-variation}
$$
where the sum is empty when $n=0$.
\end{lemma}

\begin{proof}
If $n=0$, equation \eqref{eq:Q-derivative} shows that
$(Q_1^{(\beta)})'(x)>0$ for $-1<x<1$; hence
\eqref{eq:ell-variation} and \eqref{eq:Q-endpoints} give
$\ell_0^{(\beta)}=1$.  Assume henceforth that $n\geq1$ and set
$$
a_k
:=
\left|
Q_{n+1}^{(\beta)}
\bigl(\xi_{k,n+1}^{(\beta)}\bigr)
\right|,
\quad 1\leq k\leq n.
$$
By \eqref{eq:critical-phase}, the Pr\"ufer phase at
$\xi_{k,n+1}^{(\beta)}$ is $k\pi$, and hence
$$
Q_{n+1}^{(\beta)}
\bigl(\xi_{k,n+1}^{(\beta)}\bigr)
=
(-1)^ka_k.
$$
As $x$ decreases from $1$ to $-1$, the successive values are
$$
1,-a_1,a_2,\ldots,(-1)^na_n,0,
$$
where the endpoint values follow from \eqref{eq:Q-endpoints}.
Since $Q_{n+1}^{(\beta)}$ is absolutely continuous, splitting the
variation in \eqref{eq:ell-variation} at these critical points gives
$$
\begin{aligned}
\ell_n^{(\beta)}
&=
(1+a_1)
+\sum_{k=1}^{n-1}(a_k+a_{k+1})
+a_n\\[7pt]
&=
1+2\sum_{k=1}^na_k.
\end{aligned}
$$
This proves \eqref{eq:ell-total-variation}.
\end{proof}

\begin{theorem}[Endpoint monotonicity]
\label{thm:endpoint-monotonicity}
If $\beta\geq-1/3$, then
$$
\ell_{n+1}^{(\beta)}>\ell_n^{(\beta)},
\quad n\geq0.
$$
\end{theorem}

\begin{proof}
Subtracting \eqref{eq:ell-total-variation} at indices $n+1$ and $n$
gives
$$
\begin{aligned}
\ell_{n+1}^{(\beta)}-\ell_n^{(\beta)}
&=
2\sum_{k=1}^n
\left\{
\left|
Q_{n+2}^{(\beta)}
\bigl(\xi_{k,n+2}^{(\beta)}\bigr)
\right|
-
\left|
Q_{n+1}^{(\beta)}
\bigl(\xi_{k,n+1}^{(\beta)}\bigr)
\right|
\right\}\\[7pt]
&\quad
+2\left|
Q_{n+2}^{(\beta)}
\bigl(\xi_{n+1,n+2}^{(\beta)}\bigr)
\right|.
\end{aligned}
\numberthis{eq:ell-difference}
$$
Every term in the sum in \eqref{eq:ell-difference} is positive by
Corollary~\ref{cor:extrema}, and the final term is positive by
\eqref{eq:critical-phase} and \eqref{eq:radius-definition}.  The sum
is empty when $n=0$, so this case is included.
\end{proof}

On a subinterval of the parameter range, the endpoint value equals
the global Lebesgue constant.  We include the short argument because it marks
the precise distinction between the endpoint and global results.

\begin{proposition}[Endpoint maximality]
\label{prop:endpoint-maximality}
If $-1/2\leq\beta\leq0$, then, for every $n\geq0$,
$$
\Lambda_n^{(0,\beta)}
=
\lambda_n^{(0,\beta)}(1)
=
\ell_n^{(\beta)}.
\numberthis{eq:endpoint-maximality}
$$
\end{proposition}

\begin{proof}
Put $d\mu_\beta(t)=(1+t)^\beta dt$.  In the range
$-1/2\leq\beta\leq0$, Gasper's positive Jacobi product formula
\cite{Gasper1971} provides a measurable probability kernel
$(x,t)\mapsto\nu_{x,t}$ on $[-1,1]$ such that
$$
P_j^{(0,\beta)}(x)P_j^{(0,\beta)}(t)
=
\int_{-1}^1P_j^{(0,\beta)}(u)d\nu_{x,t}(u).
\numberthis{eq:Gasper-product}
$$
Let
$$
(\tau_x f)(t):=\int_{-1}^1f(u)d\nu_{x,t}(u).
$$
For each $j\geq0$, \eqref{eq:Gasper-product} and orthogonality give
$$
\begin{aligned}
\int_{-1}^1
(\tau_xP_j^{(0,\beta)})(t)d\mu_\beta(t)
&=
P_j^{(0,\beta)}(x)
\int_{-1}^1P_j^{(0,\beta)}(t)d\mu_\beta(t)\\[7pt]
&=
\int_{-1}^1P_j^{(0,\beta)}(t)d\mu_\beta(t).
\end{aligned}
$$
Indeed, for $j=0$ both sides are the mass of $\mu_\beta$, while for
$j\geq1$ both sides vanish.  Thus the Jacobi measure is invariant
under $\tau_x$ on polynomials.  Since $\tau_x$ is a positive
contraction, uniform polynomial approximation gives
$$
\int_{-1}^1(\tau_xf)(t)d\mu_\beta(t)
=
\int_{-1}^1f(t)d\mu_\beta(t)
\numberthis{eq:translation-invariance}
$$
for every continuous $f$; see also
\cite[Sections~1.4 and 3.2]{Bavinck}.  Summing
\eqref{eq:Gasper-product} with the coefficients in
\eqref{eq:jacobi-kernel} gives
$$
K_n^{(0,\beta)}(x,t)
=
\int_{-1}^1
K_n^{(0,\beta)}(1,u)d\nu_{x,t}(u).
\numberthis{eq:kernel-translation}
$$
The triangle inequality in \eqref{eq:kernel-translation}, Tonelli's
theorem, \eqref{eq:translation-invariance}, and the definitions
\eqref{eq:lebesgue-definitions} and \eqref{eq:ell-definition}, applied
to the continuous function $|K_n^{(0,\beta)}(1,\cdot)|$,
consequently give
$$
\begin{aligned}
\lambda_n^{(0,\beta)}(x)
&\leq
\int_{-1}^1\int_{-1}^1
\bigl|K_n^{(0,\beta)}(1,u)\bigr|
d\nu_{x,t}(u)(1+t)^\beta dt\\[7pt]
&=
\int_{-1}^1
\bigl|K_n^{(0,\beta)}(1,u)\bigr|(1+u)^\beta du
=
\ell_n^{(\beta)}.
\end{aligned}
\numberthis{eq:endpoint-bound}
$$
Equality holds at $x=1$ by \eqref{eq:ell-definition}, and
\eqref{eq:endpoint-bound} and \eqref{eq:lebesgue-definitions} yield
\eqref{eq:endpoint-maximality}.
\end{proof}

\begin{theorem}[Global Jacobi Lebesgue constants]
\label{thm:global-monotonicity}
If $-1/3\leq\beta\leq0$, then
$$
\Lambda_{n+1}^{(0,\beta)}
>
\Lambda_n^{(0,\beta)},
\quad n\geq0.
$$
The same conclusion holds for the sequence
$(\Lambda_n^{(\gamma,0)})_{n\geq0}$ when
$-1/3\leq\gamma\leq0$.
\end{theorem}

\begin{proof}
For $(0,\beta)$, combine Theorem~\ref{thm:endpoint-monotonicity}
with \eqref{eq:endpoint-maximality}.  For the reflected parameters,
the Jacobi identity
$$
P_j^{(\gamma,0)}(x)
=
(-1)^jP_j^{(0,\gamma)}(-x)
$$
first gives, after the change of variables $u=-t$ in the defining
norm, $h_j^{(\gamma,0)}=h_j^{(0,\gamma)}$.  Combining this equality
with \eqref{eq:jacobi-kernel} gives the kernel identity in the second
line below.  Finally, \eqref{eq:jacobi-weight} gives
$w_{\gamma,0}(t)=w_{0,\gamma}(-t)$, so the same change of variables
in \eqref{eq:lebesgue-definitions} gives the third line:
$$
\begin{aligned}
h_j^{(\gamma,0)}
&=
h_j^{(0,\gamma)},\\[7pt]
K_n^{(\gamma,0)}(x,t)
&=
K_n^{(0,\gamma)}(-x,-t),\\[7pt]
\lambda_n^{(\gamma,0)}(x)
&=
\lambda_n^{(0,\gamma)}(-x).
\end{aligned}
\numberthis{eq:reflection-chain}
$$
Taking suprema in the last identity in \eqref{eq:reflection-chain}
shows that
$\Lambda_n^{(\gamma,0)}=\Lambda_n^{(0,\gamma)}$.  The asserted
monotonicity now follows from the already proved $(0,\beta)$ case
with $\beta=\gamma$.
\end{proof}

For $\beta>0$, Theorem~\ref{thm:endpoint-monotonicity} remains valid,
but Proposition~\ref{prop:endpoint-maximality} is not available in
this orientation.  Accordingly, we assert monotonicity of the
endpoint Lebesgue function there, not of the global constant.

\subsection*{The Legendre case and the Qu--Wong desideratum}

At $\beta=0$, write $P_j=P_j^{(0,0)}$ and
$L_n=\Lambda_n^{(0,0)}$.  Equations
\eqref{eq:legendre-contiguous} and \eqref{eq:Q-derivative} reduce to
$$
Q_m^{(0)}
=
P_m^{(0,-1)}
=
\frac{P_m+P_{m-1}}2,
\quad
\bigl(Q_m^{(0)}\bigr)'
=
\frac m2P_{m-1}^{(1,0)}.
\numberthis{eq:legendre-radau}
$$
In view of \eqref{eq:legendre-radau},
Corollary~\ref{cor:extrema} is precisely the theorem of Wong
and Zhang on corresponding relative extrema \cite{WongZhang}, now
obtained from the continuous equal-phase comparison.

Equations \eqref{eq:ell-variation}, \eqref{eq:ell-total-variation},
and \eqref{eq:endpoint-maximality} give
$$
L_n
=
\int_{-1}^1
\left|\bigl(Q_{n+1}^{(0)}\bigr)'(x)\right|dx
=
1+2\sum_{k=1}^n
\left|
Q_{n+1}^{(0)}\bigl(\xi_{k,n+1}^{(0)}\bigr)
\right|.
\numberthis{eq:legendre-representation}
$$
Theorem~\ref{thm:endpoint-monotonicity} therefore gives the following
short consequence.

\begin{corollary}[Qu--Wong]
The Legendre Lebesgue constants satisfy
$$
L_{n+1}>L_n,
\quad n\geq0.
$$
\end{corollary}

Representation \eqref{eq:legendre-representation} is of the kind
sought in the conclusion of \cite{QuWong}; the termwise inequality in
Corollary~\ref{cor:extrema} makes its
monotonicity immediate.  Thus no asymptotic expansion, explicit
error bound, or finite computation enters this proof.

\section[Why the direct method closes at alpha zero]
{Why the direct method closes at
\texorpdfstring{$\alpha=0$}{alpha zero}}
\label{sec:general}

We finally locate the obstruction to applying the same
two-dimensional Pr\"ufer comparison to arbitrary Jacobi parameters.
The discussion also clarifies what does and does not follow from the
present result.

For general $\alpha,\beta>-1$, the Jacobi norm and the endpoint form
of the Christoffel--Darboux formula give
$$
\begin{aligned}
K_n^{(\alpha,\beta)}(1,x)
&=
c_n^{(\alpha,\beta)}P_n^{(\alpha+1,\beta)}(x),\\[7pt]
c_n^{(\alpha,\beta)}
&:=
\frac{\Gamma(n+\alpha+\beta+2)}
{2^{\alpha+\beta+1}\Gamma(\alpha+1)\Gamma(n+\beta+1)}.
\end{aligned}
\numberthis{eq:general-endpoint-kernel}
$$
Define the cumulative endpoint kernel
$$
F_{n+1}^{(\alpha,\beta)}(x)
:=
c_n^{(\alpha,\beta)}
\int_{-1}^x
w_{\alpha,\beta}(t)P_n^{(\alpha+1,\beta)}(t)dt.
\numberthis{eq:F-definition}
$$
Then $F_{n+1}^{(\alpha,\beta)}$ is absolutely continuous and
$$
\begin{aligned}
F_{n+1}^{(\alpha,\beta)}(-1)&=0,\\[7pt]
F_{n+1}^{(\alpha,\beta)}(1)&=1,\\[7pt]
\bigl(F_{n+1}^{(\alpha,\beta)}\bigr)'(x)
&=
w_{\alpha,\beta}(x)K_n^{(\alpha,\beta)}(1,x),
\quad -1<x<1.
\end{aligned}
\numberthis{eq:F-data}
$$
By \eqref{eq:F-definition} and \eqref{eq:general-endpoint-kernel},
the middle identity in \eqref{eq:F-data} follows from
$$
F_{n+1}^{(\alpha,\beta)}(1)
=
\int_{-1}^1K_n^{(\alpha,\beta)}(1,t)
w_{\alpha,\beta}(t)dt
=1,
$$
because the kernel in \eqref{eq:jacobi-kernel} reproduces the
constant polynomial.  By \eqref{eq:jacobi-weight},
\eqref{eq:lebesgue-definitions}, and \eqref{eq:F-data}, the endpoint
value $\lambda_n^{(\alpha,\beta)}(1)$ of the Lebesgue function is
therefore always the total variation of this
cumulative kernel.  What is special about
$\alpha=0$ is that this primitive is precisely
$Q_{n+1}^{(\beta)}$, which has the explicit Jacobi--Radau form and
satisfies a closed homogeneous second-order equation, as shown in
Section~\ref{sec:radau}.

Here is the obstruction in operator form.  Write
$$
F:=F_{n+1}^{(\alpha,\beta)},
\quad
v:=F',
\quad
\sigma:=\alpha+\beta,
\quad
m:=n+1.
\numberthis{eq:general-notation}
$$
Put
$$
p(x):=P_n^{(\alpha+1,\beta)}(x),
\quad
h(x):=\frac{w_{\alpha,\beta}'(x)}{w_{\alpha,\beta}(x)}
=-\frac{\alpha}{1-x}+\frac{\beta}{1+x},
$$
and
$$
B_p(x):=\beta-\alpha-1-(\sigma+3)x.
$$
The Jacobi equation for $p$ is
$$
(1-x^2)p''+B_pp'+n(n+\sigma+2)p=0.
\numberthis{eq:general-jacobi-equation}
$$
By \eqref{eq:F-data} and \eqref{eq:general-endpoint-kernel},
$v=c_n^{(\alpha,\beta)}w_{\alpha,\beta}p$; direct differentiation
gives
$$
\begin{aligned}
p
&=
\frac{v}{c_n^{(\alpha,\beta)}w_{\alpha,\beta}},\\[7pt]
p'
&=
\frac{v'-hv}{c_n^{(\alpha,\beta)}w_{\alpha,\beta}},\\[7pt]
p''
&=
\frac{v''-2hv'+(h^2-h')v}
{c_n^{(\alpha,\beta)}w_{\alpha,\beta}}.
\end{aligned}
\numberthis{eq:general-substitution}
$$
Substitution of \eqref{eq:general-substitution} into
\eqref{eq:general-jacobi-equation} yields
$$
\begin{aligned}
(1-x^2)v''
&+\bigl[B_p-2(1-x^2)h\bigr]v'\\[7pt]
&+\bigl[
(1-x^2)(h^2-h')-B_ph+n(n+\sigma+2)
\bigr]v
=0.
\end{aligned}
\numberthis{eq:v-intermediate}
$$
The two coefficients simplify as follows:
$$
\begin{aligned}
(1-x^2)h
&=
(\beta-\alpha)-\sigma x,\\[7pt]
B_p-2(1-x^2)h
&=
\alpha-\beta-1+(\sigma-3)x,\\[7pt]
(1-x^2)(h^2-h')-B_ph
&=
2\sigma-\frac{2\alpha}{1-x},\\[7pt]
n(n+\sigma+2)+2\sigma
&=
(n+2)(n+\sigma).
\end{aligned}
\numberthis{eq:v-coefficients}
$$
Combining \eqref{eq:v-intermediate} and
\eqref{eq:v-coefficients} gives
$$
\begin{aligned}
(1-x^2)v''
&+\bigl[\alpha-\beta-1+(\sigma-3)x\bigr]v'\\[7pt]
&+\left[
(n+2)(n+\sigma)-\frac{2\alpha}{1-x}
\right]v
=0,
\quad -1<x<1.
\end{aligned}
\numberthis{eq:v-equation}
$$
Introduce
$$
\begin{aligned}
\mathcal D_m F
&:=
(1-x^2)F''
+\bigl[\alpha-\beta-1+(\sigma-1)x\bigr]F'\\[7pt]
&\quad +m(m+\sigma)F.
\end{aligned}
\numberthis{eq:D-definition}
$$
If
$$
A(x):=\alpha-\beta-1+(\sigma-1)x,
$$
then differentiating \eqref{eq:D-definition}, using $F'=v$, gives
$$
\begin{aligned}
\bigl(\mathcal D_mF\bigr)'
&=
(1-x^2)v''+\bigl[A(x)-2x\bigr]v'\\[7pt]
&\quad
+\bigl[A'(x)+m(m+\sigma)\bigr]v.
\end{aligned}
\numberthis{eq:D-differentiated}
$$
The coefficients in \eqref{eq:D-differentiated} satisfy, by
\eqref{eq:general-notation},
$$
\begin{aligned}
A(x)-2x
&=
\alpha-\beta-1+(\sigma-3)x,\\[7pt]
A'(x)+m(m+\sigma)
&=
\sigma-1+(n+1)(n+\sigma+1)\\[7pt]
&=
(n+2)(n+\sigma).
\end{aligned}
$$
Substitution of \eqref{eq:v-equation} into
\eqref{eq:D-differentiated} therefore yields the exact identity
$$
\bigl(\mathcal D_m F\bigr)'
=
\frac{2\alpha}{1-x}F',
\quad -1<x<1.
\numberthis{eq:D-obstruction}
$$
For $\alpha=0$, the right-hand side vanishes.  In that case the
constant in \eqref{eq:general-endpoint-kernel} becomes
$(n+\beta+1)/2^{\beta+1}$.  Hence
\eqref{eq:Q-definition}, \eqref{eq:weighted-differentiation}, and the
first identity in \eqref{eq:F-data} identify $F$ with
$Q_m^{(\beta)}$.
Equation \eqref{eq:Q-equation} then gives
$\mathcal D_mF=0$.  For $\alpha\ne0$, the singular term on the
right-hand side of \eqref{eq:D-obstruction} is unavoidable in this
reduction.

More explicitly, suppose one attempts to obtain \eqref{eq:v-equation}
by differentiating an equation
$$
(1-x^2)F''+B(x)F'+\lambda F=0
\numberthis{eq:hypothetical-F-equation}
$$
with constant spectral parameter $\lambda$.  Differentiating
\eqref{eq:hypothetical-F-equation} gives
$$
(1-x^2)v''+\bigl[B(x)-2x\bigr]v'
+\bigl[B'(x)+\lambda\bigr]v=0.
\numberthis{eq:hypothetical-v-equation}
$$
Comparison of the coefficient of $v'$ in
\eqref{eq:hypothetical-v-equation} with that in
\eqref{eq:v-equation} forces
$$
B(x)
=
\alpha-\beta-1+(\sigma-1)x.
$$
The coefficient $B'+\lambda$ in
\eqref{eq:hypothetical-v-equation} would then be constant, whereas
the coefficient of $v$ in \eqref{eq:v-equation} contains
$-2\alpha/(1-x)$.  The two equations can therefore agree only when
$\alpha=0$; in that case equality forces
$\lambda=m(m+\sigma)$.  This is a structural failure of the present
constant-frequency, two-dimensional Pr\"ufer mechanism; it is not a
claim that all discrete or global monotonicity phenomena must fail
when $\alpha\ne0$.

The same point is visible in the weighted Jacobi identity
$$
\frac{d}{dx}
\left\{
(1+x)^{\beta+1}P_n^{(\alpha,\beta+1)}(x)
\right\}
=
(n+\beta+1)(1+x)^\beta
P_n^{(\alpha+1,\beta)}(x),
\quad -1<x<1.
\numberthis{eq:general-weighted-differentiation}
$$
For $\alpha=0$, the right-hand side of
\eqref{eq:general-weighted-differentiation}, after multiplication by
$2^{-\beta-1}$, is the weighted endpoint kernel in
\eqref{eq:general-endpoint-kernel} and yields
$Q_{n+1}^{(\beta)}$ by \eqref{eq:Q-definition}.  When
$\alpha\ne0$, the Jacobi weight in \eqref{eq:jacobi-weight} contains
the additional factor
$(1-x)^\alpha$; placing it inside the derivative produces exactly the
additional term $2\alpha F'/(1-x)$ in
\eqref{eq:D-obstruction}.

\section{Concluding remarks}

The proof yields three conclusions.  First, throughout
$\beta\geq-1/3$, consecutive weighted Jacobi--Radau functions are
strictly ordered at every common phase.  Second, restriction to
$s=k\pi$ orders the corresponding relative extrema term by term.
Third, total variation shows that the endpoint sequence
$(\lambda_n^{(0,\beta)}(1))_{n\geq0}$ is strictly increasing.  The
positive product formula yields the same conclusion for the global
sequences
$(\Lambda_n^{(0,\beta)})_{n\geq0}$ for
$-1/3\leq\beta\leq0$ and, by reflection, to
$(\Lambda_n^{(\beta,0)})_{n\geq0}$ on the same interval.

At $\beta=0$, this chain supplies a direct proof of the Wong--Zhang
extremal theorem and then of the Qu--Wong theorem on Szeg\H{o}'s
conjecture.  The equal-phase Pr\"ufer mechanism used here is based on
ideas developed in an earlier preprint of the first author
\cite{Castillo2026}.  In that
precise sense, the present argument carries out the
alternative-expression programme proposed by Qu and Wong: the
finite total-variation representation exposes the relevant
quantities, and the equal-phase comparison proves each required
inequality before they are summed.

\section*{Acknowledgements}
The first author was supported by the Centre for Mathematics of the
University of Coimbra (CMUC), funded by the Portuguese Foundation for
Science and Technology (FCT), through the projects
\href{https://doi.org/10.54499/UID/00324/2025}{UID/00324/2025} and
UID/PRR/00324/2025, and by FCT grant
\href{https://doi.org/10.54499/2022.00143.CEECIND/CP1714/CT0002}
{2022.00143.CEECIND}.

\end{document}